\documentclass[11pt]{amsart}
\usepackage{amssymb,amstext,amsmath,amscd,amsthm,amsfonts,enumerate,latexsym}
\usepackage{extarrows}
\newtheorem{theorem}{Theorem}[section]
\newtheorem{lemma}[theorem]{Lemma}
\newtheorem{corollary}[theorem]{Corollary}
\newtheorem{proposition}[theorem]{Proposition}
\theoremstyle{definition}
\newtheorem{definition}[theorem]{Definition}

\newtheorem{remark}[theorem]{Remark}

\newtheorem*{claim}{Claim}
\numberwithin{equation}{section}
 
\title[Krull--Gabriel dimension of Cohen--Macaulay modules]{Krull--Gabriel dimension of Cohen--Macaulay modules over hypersurfaces of countable Cohen--Macaulay representation type}
\date{\today}
\thanks{The author was partly supported by JSPS KAKENHI Grant Number 18K13399 and 21K03213.}
\author{Naoya Hiramatsu} 
\address{General Education Program, National Institute of Technology, Kure College, 2-2-11, Agaminami, Kure Hiroshima, 737-8506 Japan}
\email{hiramatsu@kure-nct.ac.jp} 

\begin{document}
%%%%%%%%%%%%%%%%%%%%%%%%%%%%%%%%%%
%%%%%%%%%%%%%%%%%%%%%%%%%%%%%%%%%%
%  Original notation
%%%%%%%%%%%%%%%%%%%%%%%%%%%%%%%%%%
%%%%%%%%%%%%%%%%%%%%%%%%%%%%%%%%%%
\def\AA{\mathcal A} 
\def\Ass{\mathrm Ass} 
\def\BB{\mathcal B} 
\def\CC{\mathcal C} 
\def\CMR{{\mathcal C}(R)} 
\def\CMRs{{\mathcal C} (R^{\sharp})}
\def\CMRsG{{\mathcal C} (R^{\sharp}\ast G)}
\def\CM{{\mathcal{C}}}
\def\Coker{\mathrm{Coker}}
\def\DD{\mathcal D} 
\def\depth{\mathrm{depth}}
\def\End{\mathrm{End}}
\def\Ext{\mathrm{Ext}}
\def\FF{\mathcal{F}}
\def\Hom{\mathrm{Hom}}
\def\HomR{\mathrm{Hom}_R}
\def\Im{\mathrm{Im}}
\def\Ind{\mathrm{Ind}}
\def\Ker{\mathrm{Ker}}
\def\KGdim{\mathrm{KGdim}\ }
\def\m{\mathfrak m}
\def\Mod{\mathrm{Mod}}
\def\mod{\mathrm{mod}}
\def\modC{{\mathrm{mod}}(\CMR )}
\def\PP{\mathcal P} 
\def\p{\mathfrak p}
\def\q{\mathfrak q}
\def\Rs{R^{\sharp }}
\def\Rss{R^{\sharp \sharp}}
\def\RsG{R^{\sharp } \ast G}
\def\Spec{\mathrm{Spec}} 
\def\SS{\mathcal S} 
\def\Ss{S^{\sharp }}
\def\sCM{\underline{{\mathcal C}}} 
\def\sCMR{\underline{{\mathcal C}}(R)} 
\def\sCMRs{\underline{{\mathcal C}}(\Rs)} 
\def\sEnd{\underline{\mathrm{End}}}
\def\sHom{\underline{\mathrm{Hom}}}
\def\sHomR{\underline{\mathrm{Hom}}_R}
\def\sHomRs{\underline{\mathrm{Hom}}_{\Rs}}
\def\smod{\underline{\mathrm{mod}}}
\def\smodC{\underline{\mathrm{mod}}(\CMR )}
\def\smodCs{\underline{\mathrm{mod}}(\CMRs )}
\def\smodCss{\underline{\mathrm{mod}}({\mathcal C} (R^{\sharp \sharp}) )}
\def\Supp{{\mathrm{Supp}}}
\def\syz{\mathrm{syz}}
\def\Z{\mathbb Z}
%%%%%%%%%%%%%%%%%% Abstract%%%%%%%%%%%%%%%%%
\begin{abstract}
We calculate the Krull--Gabriel dimension of the functor category of the (stable) category of maximal Cohen--Macaulay modules over hypersurfaces of countable Cohen--Macaulay representation type. 
We show that the Krull--Gabriel dimension is $0$ if the hypersurface is of finite Cohen--Macaulay representation type and that is $2$ if the hypersurface is of countable but not finite Cohen--Macaulay representation type.  
\end{abstract}
%%%%%%%%%%%%%%%%%%%%%%%%%%%%%%%%%%
%%%%%%%%%%%%%%%%%%%%%%%%%%%%%%%%%%
\maketitle

\section{Introduction}\label{sec1}
The notion of Krull--Gabriel dimension has been considered under a functorial approach viewpoint of representation theory of finite dimensional algebras. 
It was introduced by Gabriel\cite{Ga62} and has been studied by many authors including Geigle\cite{G85} and Schr\"oer\cite{S00}. 
The notion is considered for an abelian category and defined by a length of a certain filtration of Serre subcategories. 
See Section \ref{A} for the precise definition. 
Let $A$ be a finite dimensional algebra and $\mod (A)$ a category of finitely generated $A$-modules. 
The functor category $\mod (\mod A)$ of $\mod A$ is an abelian category, so that the Krull-Gabriel dimension of  $\mod (\mod (A))$, which is denoted by $\KGdim  \mod (\mod A)$, can be investigated.   
The Krull--Gabriel dimension is closely related to representation types of algebras. 
It was proved by Auslander\cite{A80} that $A$ is of finite representation type if and only if $\KGdim \mod (\mod (A)) = 0$. 
Krause\cite{Kr98} shows that there are no algebras such that $\KGdim \mod (\mod A) \not= 1$ and Geigle\cite{G85} shows that every tame hereditary algebra is of Krull--Gabriel dimension $2$. 
Geigle\cite{G85} also shows that an algebra which is of wild representation type has Krull--Gabriel dimension $\infty$. 

Let $R$ be a commutative Cohen--Macaulay local ring and $\CMR$ the category of maximal Cohen--Macaulay $R$-modules. 
In this paper, we study the Krull--Gabriel dimension of the functor category of $\CMR$. 
More precisely we calculate the Krull--Gabriel dimension of $\smodC$; the category of finitely presented contravariant additive functors $F$ with $F(R)=0$ from $\CMR$ to a category of abelian groups. 
First, we shall show the following theorem, which gives an analogy of a result due to Auslander.

\begin{theorem}\label{Z1}[Theorem \ref{A4}]
Let $R$ be a complete Cohen--Macaulay local ring. 
Then $R$ is of finite Cohen--Macaulay representation type if and only if $\KGdim \smodC = 0$. 
\end{theorem}

Let $k$ be an algebraically closed uncountable field of characteristic not two.
Next, we investigate the case when $R$ is a hypersurface that is of countable but not finite Cohen--Macaulay representation type. 
Namely $R$ is isomorphic to the ring $k[\![x_0 , x_1 , x_2 , \dots, x_n]\!]/(f)$, where $f$ is of the following:
$$
f=\begin{cases}
x_0 ^2+x_2^2+\cdots+x_n^2 & (A_\infty),\\
x_0^2x_1+x_2^2+\cdots+x_n^2 & (D_\infty).
\end{cases}
$$

\begin{theorem}\label{Z2}[Corollary \ref{C11}]
Let $k$ be an algebraically closed uncountable field of characteristic not two.
Let $R$ be a hypersurface of countable but not finite Cohen--Macaulay representation type. 
Then $\KGdim \smodC =2$. 
\end{theorem}

The studies of Krull--Gabriel dimension of maximal Cohen--Macaulay modules over 1-dimensional hypersurfaces of type $(A_{\infty})$ and $(D_{\infty})$ are given by Puninski\cite{P18} and by Los and Puninski\cite{LP19}. 
Their studies investigate the Krull--Gabriel dimension of the definable category of maximal Cohen--Macaulay modules in the category of all $R$-modules so that our studies are different from theirs. 

The organization of this paper is as follows.
In Section \ref{A}, we review the notion of Krull--Gabriel dimension and consider the case when $R$ is of finite Cohen--Macaulay representation type (Theorem \ref{A4}).
Section \ref{B} is devoted to compute the Krull--Gabriel dimension of $\smodC$ where $R$ is a 1-dimensional hypersurface singularities of type ($A_{\infty} $). 
In Section \ref{D}, we calculate the Krull--Gabriel dimension over a 2-dimensional hypersurface of type ($D_{\infty} $).
In Section \ref{C} we investigate how Krull--Gabriel dimension changes with Kn\"orrer's periodicity (Theorem \ref{C10}). 
Using it we attempt to compute the Krull--Gabriel dimension for higher (or lower) dimensional cases (Corollary \ref{C11}).

%%%%%%%%%%%%%%%%%%%%%%%%%%%%%%%%%%%%%%%%%
%%%%%%%%%%%%%%%%%%%%%%%%%%%%%%%%%%%%%%%%%
%%%%%%%%%%%%%%%%%%%%%%%%%%%%%%%%%%%%%%%%%
\section{Preliminaries}\label{A}

Let us recall the definition of Krull--Gabriel dimension for an abelian category. 
Let $\AA$ be an abelian category. 
We say that a full subcategory $\SS$ of $\AA$ is a Serre subcategory if $\SS$ is closed under taking subobjects, quotients, and extensions.

\begin{definition}\cite[Definition 2.1]{G85}\label{A1}
Let $\AA$ be an abelian category. 
Define $\AA_{-1} = 0$. 
For each $n \geq 0$, let $\AA _{n}$ be the category of all objects which are finite length in $\AA/\AA_{n-1}$. 
We define $\KGdim \AA = min \{ n \ \verb+|+ \ \AA = \AA _n \}$ if such a minimum exists, and $\KGdim \AA = \infty $ else. 
For an object $X$ in $\AA$, we define $\KGdim X$ by a minimum number $n$ such that $X$ is in $\AA_n$. 
\end{definition}

To show a simpleness of an object in a quotient category, the following lemma is useful.

\begin{lemma}\cite[Lemma 1.1]{GR74}\label{A2}
Let $\AA$ be an abelian category and $\SS$ the Serre subcategory.  
For an exact sequence $0 \to N \to M \to L \to 0$ in $\AA / \SS$, there is an exact sequence $0 \to N' \to M \to L' \to 0$ in $\AA$ such that $N \cong N'$ and $L \cong L'$ in $\AA/ \SS$. 
Therefore the object $X$ of $\AA$ becomes simple in $\AA/ \SS$ if $X$ is not an object of $\SS$ and if for each subobject $V$ of $X$ either $V$ or $X/V$ belongs to $\SS$. 
\end{lemma}

Let $R$ be a commutative Noetherian ring with a finite Krull dimension. 
We denote by $\mod (R)$ a category of finitely generated $R$-modules with $R$-homomorphisms. 
We compute the Krull--Gabriel dimension of $\mod (R)$.

\begin{lemma}\label{A8}
Let $\AA$ be an abelian category and $\SS$, $\SS '$ the Serre subcategories with $\SS ' \subseteq \SS$. 
Suppose that $M \cong N$ in $\AA / \SS '$. 
Then $M \in \SS$ if and only if $N \in \SS$. 
\end{lemma}

\begin{proof}
Since $M \cong N$ in $\AA / \SS '$, there is a morphism $f \in \Hom _{\AA} (M' , N/N')$ where $M/M'$, $N' \in \SS '$. 
Then $f$ is a pseudo-isomorphism in $\AA$, that is $\Ker f$ and $\Coker f$ belong to $\SS '$ (cf. \cite[Lemma 4, p. 367]{Ga62}). 
Let $N$ belong to $\SS$. 
The quotient module $N/N'$ belongs to $\SS$. 
One can also show that $\Ker f$ and $\Coker f$ belong to $\SS$ since $\SS ' \subseteq \SS$.
Thus $M'$ is contained in $\SS$. 
Since $M/M'$ belongs to $\SS$ we have $M \in \SS$. 
The converse holds by the same argument. 
\end{proof}

\begin{proposition}\label{A3}
Let $R$ be a commutative Noetherian ring with a finite Krull dimension. 
Then $\KGdim \mod (R) = \dim R$. 
\end{proposition}

\begin{proof}
We denote by $\SS_i$ the subcategory of $\mod (R)$ consisting of all finitely generated $R$-modules $M$ with $\dim M \leq i$. 
Notice that $\SS _i$ is a Serre subcategory (see \cite[Example 4.5.(9)]{T08}). 
We shall show $R/\p$ is a simple object in $\mod (R) / \mod (R) _{i-1}$ for a prime ideal $\p$ with $\dim R/\p = i$ and $\mod (R)_i = \SS _i$. 
We prove it by induction on $i$. 
First, for a maximal ideal $\m$, $R/\m$ is a field, so that it is simple in $\mod (R)$. 
We also remark that a finitely generated $R$-module $M$ has finite length iff $\dim M = 0$. 
Thus $\mod (R) _0 = \SS_0$. 
Let $i>0$ and $\p$ be a prime ideal $\p$ with $\dim R/\p = i$. 
We consider an exact sequence $0 \to V \to R/\p \to C \to 0$ in $\mod (R)$. 
Localizing by $\p$, we have $0 \to V_{\p} \to ( R/\p )_{\p} \to C_{\p} \to 0$. 
Since $( R/\p )_{\p}$ is a field, $V_{\p}$ or $C_{\p} = 0$. 
One has $\Ass (V) \subseteq \Ass (R/\p ) = \{ \p \}$, so that $V_{\p} \not=0$. 
Thus we have $C_{\p} =0$. 
Let $\q \in \Supp (C)$. 
Then $(R/ \p)_{\q} \not=0$, thus $\p \subseteq \q$. 
Since $C_{\p}=0$, $\mathrm{ann} (C) \not \subseteq \p$. 
Hence $\p \subsetneq \q$. 
It yields that $\dim C < \dim R/\p =i$, so that $C \in \SS_{i-1}$. 
By the induction hypothesis, one has $\SS_{i-1} = \mod (R) _{i-1}$. 
Consequently, $R/\p$ is simple in $\mod (R)/ \mod (R)_{i-1}$ by Lemma \ref{A2}. 
Suppose that $S$ is a simple object in $\mod (R) / \mod (R) _{i-1}$. 
According to \cite[Theorem 4.1, Example 4.5.(9)]{T08}, a finitely generated $R$-module $M$ belongs to $\SS_i$ iff $\Ass (M)$ is contained in $\PP_i = \{ \p \in \Spec (R) \ \verb+|+\ \dim R/\p \leq i \}$. 
Since $S$ is not in $\mod (R) _{i-1}$, $\Ass (S)$ is not contained in $\PP_{i-1}$. 
For a prime ideal $\p$ in $\Ass (S)$, $S$ has a submodule that is isomorphic to $R/\p$. 
If the prime ideal $\p$ satisfies $\dim R/\p = i$, $S$ is isomorphic to $R/\p$ in $\mod (R) / \mod (R) _{i-1}$ for the simpleness of $S$. 
Assume that a prime ideal $\p$ with $\dim R/\p > i$ belong to $\Ass (S)$. 
Then $R/\p $ is also isomorphic to $S$ in $\mod (R) / \mod (R) _{i-1}$ since $R/\p \not \in \SS _{i-1} = \mod (R) _{i-1}$. 
However, since $\dim R/\p > i$, there exists a prime ideal $\q$ such that $\p \subset \q$ and $\dim R/ \q = i$. 
Moreover one has the exact sequence in $\mod (R)$: $0 \to \q/\p \to R/\p \to R/\q \to 0$. 
Notice that $\Ass (\q/\p ) \subseteq \Ass (R /\p)$. 
Particularly $\q/\p$ is not in $\mod (R) _{i-1}$. 
This concludes that $R/\p$, hence $S$, is not simple in $\mod (R) / \mod (R) _{i-1}$, and it is a contradiction. 
The observation shows that every simple object in $\mod (R) / \mod (R) _{i-1}$ is isomorphic to $R/\p$ with $\dim R/\p = i$.

Let $M$ be a finitely generated $R$-module. 
According to \cite[Theorem 6.4]{Ma}, we have a filtration of $M$ in $\mod (R)$
\begin{equation}\label{A3-1}
0= M_0 \subset M_1 \subset \cdots \subset M_n = M
\end{equation}
such that $M_k / M_{k-1} \cong R/\p _k$ with prime ideals $\p_k$. 
Suppose that $\dim M = i$. 
Then $\dim R/\p_k \leq i$ for the prime ideals $\p_k$ since $\mathrm{Ass}(M) \subseteq \{ \p_1, \cdots , \p_n \}$. 
Since $R/\p$ with $\dim R/\p = i$ is a simple object in $\mod (R)/ \mod (R)_{i-1}$, $M$ belongs to $\mod (R)_{i}$. 
This shows that $\SS_i \subseteq \mod (R)_{i}$. 
Conversely, for each $M \in \mod (R)_{i}$, we have a filtration $0= M_0 \subset M_1 \subset \cdots \subset M_n = M$ in $\mod (R)/ \mod (R)_{i-1}$ such that $M_k / M_{k-1} \cong R/\p _k$ in $\mod (R)/ \mod (R)_{i-1}$ with $\dim R/ \p_k = i$. 
By Lemma \ref{A8}, $M_k / M_{k-1}$ belongs to $\SS_{i}$ since $R/ \p_k$ is in $\SS_{i}$. 
Hence $M$ belongs to $\SS_{i}$. 
Consequently $\mod (R)_i = \SS _i$.

For each $M \in \mod (R)$, one has $\KGdim M \leq max \{ \dim R/\p_k \verb+|+ k= 1, \cdots n\}$ by the filtration (\ref{A3-1}). 
Thus $\KGdim \mod (R) \leq sup \{\dim R/\p \verb+|+ \p \in \mathrm{Spec} R \} \leq  \dim R$. 
On the other hand, take a minimal associated prime ideal $\p$ of $R$, then $\dim R/\p = \dim R$, so that $\dim R \leq \KGdim \mod (R)$. 
Therefore we obtain $\KGdim \mod (R) = \dim R$. 
\end{proof}

From now we focus on a category of maximal Cohen--Macaulay (abbr. MCM) modules. 
In the rest of the paper we always assume that $(R, \m)$ is a complete CM local ring. 
We denote by $\CMR$ the full subcategory of $\mod (R)$ consisting of all MCM $R$-modules and by $\CM_0 (R)$ the full subcategory of $\CMR$ consisting of all modules that are locally free on the punctured spectrum of $R$.

Now let us recall the full subcategory of the functor category of $\CMR$ which is called the Auslander category. 
We give a brief review of the Auslander category. 
See \cite[Chapter 4 and 13]{Y} for the details. 
The Auslander category $\mod( \CMR )$ is the category whose objects are finitely presented contravariant additive functors from $\CMR$ to a category of abelian groups and whose morphisms are natural transformations between functors. 
We denote by $\smodC$ the full subcategory $\mod( \CMR )$ consisting of functors $F$ with $F(R )= 0$. 
Note that every object $F \in \smodC$ is obtained from a short exact sequence in $\CMR$. 
Namely, we have the short exact sequence $0 \to N \to M \to L \to 0$ such that 
$$
0 \to \Hom _R (\  , N) \to \Hom _R (\  , M) \to \Hom _R (\  , L) \to F \to 0 
$$
is exact in $\mod( \CMR )$.

We denote by $\sCMR$ the stable category of $\CMR$. 
The objects of $\sCMR$ are the same as those of $\CMR$, and the morphisms of $\sCMR$ are elements of $\sHomR (M, N) = \Hom _R(M, N)/P(M, N)$ for $M, N \in \sCMR$, where $P(M, N)$ denote the set of morphisms from $M$ to $N$ factoring through free $R$-modules. 
For a finitely generated $R$-module $M$, we denote by $\syz _{R} ^1 (M)$ the reduced first syzygy of $M$.

\begin{remark}\label{A5}
\begin{enumerate}[\rm(1)]
\item Since $R$ is complete, $\CMR$, thus $\sCMR$, is a Krull-Schmidt category. 
\item The categories $\modC$ and $\smodC$ are abelian categories (cf. \cite[(4.17), (4.19)]{Y}). 
\item The category $\smodC$ is equivalent to the Auslander category $\mod (\sCMR )$ of $\sCMR$. 
See \cite[Remark 2.6]{Y05}. 
Moreover, according to \cite[Remark 4.16]{Y05}, for $F \in \modC$ with $0 \to \Hom _R (\  , N) \to \Hom _R (\  , M) \to \Hom _R (\  , L) \to F \to 0$, we have an exact sequence $\sHom _R (\  , N) \to \sHom _R (\  , M) \to \sHom _R (\  , L) \to F \to 0$. 
\item If $R$ is Gorenstein the stable category $\sCMR$ has a structure of a triangulated category with the suspension (shift) functor defined by $(-)[-1] = \syz _{R} ^1 (-)$ (cf. \cite{Ha88}). 
\end{enumerate}
\end{remark}

In the paper, we use a theory of Auslander-Reiten (abbr. AR) sequences. 
For the detail, we recommend the reader to refer to \cite{Y}. 
Let $0 \to Z \to Y \to X \to 0$  be an AR sequence in $\CMR$. 
Then the functor $S_X$  defined by an exact sequence 
$$
0 \to \Hom _R (\  , Z) \to \Hom _R (\  , Y) \to \Hom _R (\  , X) \to S_X \to 0 
$$
is a simple object in $\mod( \CMR )$  and all the simple objects in $\mod( \CMR )$ are obtained in this way from AR sequences (\cite[(4.12)]{Y}). 

For a functor $F \in \modC$, we denote by $\Supp (F)$ a set of isomorphism classes of indecomposable MCM $R$-modules $M$ with $F(M) \not =0$: 
$$
\Supp (F) = \{ M \ \verb+|+ \ F(M) \not=0 \} /\cong . 
$$ 

Let us show the first result of the paper, which is an analogical result due to Auslander.     
We say that $R$ is of finite CM representation type if there are only a finite number of isomorphism classes of indecomposable MCM $R$-modules. 

\begin{theorem}\label{A4}
Let $R$ be a complete CM local ring. 
Then $R$ is of finite CM representation type if and only if $\KGdim \smodC = 0$. 
\end{theorem}

\begin{proof}
Suppose that $R$ is of finite representation type. 
As mentioned in \cite[Chapter 13]{Y}, every functor $F \in \smodC$ has finite length. 
Hence $\KGdim \smodC = 0$. 

Conversely suppose that $\KGdim \smodC = 0$. 
According to \cite[Lemma 2.1]{H17}, there exists $X \in \CMR$ such that $\sHomR (M, X) \not= 0$ for all non free MCM $R$-modules $M$. 
That is, $\Supp (\sHomR (-, X)) \cup \{R \} = \Ind (\CMR )$. 
Since $\sHomR (-, X) \in \smodC$, $\sHomR (-, X)$ belongs to $\smodC _{0}$.   
Thus there are only finitely many indecomposable MCM $R$-modules $M$ such that $\sHomR (M, X) \not= 0$, so that $R$ is of finite CM representation type. 
\end{proof}

\begin{remark}\label{A6}
We note that the Krull--Gabriel dimension of $\modC$ is not always $0$ even if $R$ is of finite CM representation type. 
Actually let $R=k[\![x]\!]$. 
Then $\CMR = \mathrm{add} \{R \}$. Thus $R$ is of finite CM representation type.  
Since $\modC = \mod (R)$ (cf. \cite[Lemma 6.4]{MT17}), we have the equality $\KGdim \mod (R) = \dim R = 1$ by Proposition \ref{A3}.  
\end{remark}

Suppose that $R$ is a Gorenstein local ring. 
Since $\sCMR$ is a triangulated category, we can define
$$
\sHom _R( -, M[-1] ) \to \sHom _R( -, L[-1] ) \to F[-1] \to 0  
$$
for every $F \in \smodC$ with $\sHom _R( -, M) \to \sHom _R( -, L) \to F \to 0$. 
For the later use, we state a lemma.

\begin{lemma}\label{A7}
Let $R$ be a Gorenstein local ring and $S$ a simple object in $\smodC /\smodC _{n-1}$. 
Then $S[-1]$ is also simple in $\smodC /\smodC _{n-1}$. 
\end{lemma}

\begin{proof}
We prove by induction on $n$. 
Suppose that $S$ is simple in $\smodC$, that is $S$ is obtained from an AR sequence. 
Notice that an AR sequence gives an AR triangle. (\cite[Proposition 2.2]{H18}.) 
Since the shift functor $(-)[-1]$ is an auto-functor, the shift of an AR triangle is also an AR triangle so that $S[-1]$ is induced by an AR triangle (sequence). 
Hence $S[-1]$ is simple in in $\smodC$. 
Then we have $F[-1]$ has finite length in $\smodC$ if $F$ does. 

Now suppose that $S$ is a simple object in $\smodC /(\smodC )_{n-1}$. 
Let $0 \to V \to S[-1] \to C \to 0$ be an admissible exact sequence in $\smodC$. 
Apply the shift functor $(-)[1]$ to the sequence, we have $0 \to V[1] \to S \to C[1] \to 0$. 
Since $S$ is simple, $V[1]$ or $C[1]$ are in $(\smodC )_{n-1}$. 
By the induction hypothesis, $V= V[1][-1]$ or $C= C[1][-1]$ are in $(\smodC )_{n-1}$. 
This implies that $S[-1]$ is simple in $\smodC /(\smodC )_{n-1}$. 
\end{proof}

%%%%%%%%%%%%%%%%%%%%%%%%%%%%%%%%%%%%%%%%%
%%%%%%%%%%%%%%%%%%%%%%%%%%%%%%%%%%%%%%%%%
%%%%%%%%%%%%%%%%%%%%%%%%%%%%%%%%%%%%%%%%%
\section{Krull--Gabriel dimension of $\smod (\CM (k[\![x,y]\!]/(x^2)))$}\label{B}

Let $k$ be an algebraically closed uncountable field of characteristic not two and $R$ a 1-dimensional hypersurface of type $(A_\infty)$, that is, $R=k[\![x, y]\!]/(x^2)$.
This section is devoted to calculate the Krull-Gabriel dimension of $\smodC$. 
It is known that $R$ is of  {\em countable} CM representation type, namely there exist infinitely but only countably many isomorphism classes of indecomposable MCM $R$-modules.
The non free indecomposable MCM $R$-modules are as follows: 
$$
I= \Coker  (x): R \to R, \qquad 
I_n = \Coker \left( \begin{smallmatrix} x & y^n \\ 0& x \end{smallmatrix}\right): R^{\oplus 2} \to R^{\oplus 2} \quad (n \geq 1). 
$$
See \cite[Proposition 4.1]{BGS87}. 
First we state the main result in this section.

\begin{theorem}\label{B1}
Let $k$ be an algebraically closed uncountable field of characteristic not $2$ and $R=k[\![x, y]\!]/(x^2)$.
Then $\KGdim \smodC = 2$. 
\end{theorem}

To prove the theorem, we shall do some preparation. 

\begin{proposition}\label{B9}\cite[Proposition 2.14 (1)]{H18}
Let $0 \to Z \to Y \to X \to 0$ be an AR sequence. 
Then the following equality holds for each indecomposable $U \in \CMR$: 
$$
\dim _k \sHomR (U, X ) + \dim _k \sHomR(U, Z) - \dim _k \sHomR(U, Y) = \underline{\mu } (U, X) + \underline{\mu } (U, X[-1]). 
$$
Here $\underline{\mu } (U, X)$ the multiplicity of $U$ as a direct summand of $X$ in $\sCMR$. 
\end{proposition}

\begin{lemma}\label{B2}
Let $R$, $I$ and $I_n$ be as above. 
The following statements hold.  
\begin{enumerate}[\rm(1)]
\item $\dim _k \sHomR (I_m, I_n) = \left\{ \begin{array}{l} 2n  \quad m \geq n, \\2m \quad m \leq n. \end{array}\right.$
\item $\dim _k \sHomR (I, I_n) = \dim _k \sHomR (I_n, I) = n$ for $n \geq 1$. 
\item $\dim _k \sHomR (I, I) = \infty$. 
\end{enumerate}
\end{lemma}

\begin{proof}
First we shall compute $\dim _k \sHomR (-, I_1)$. 
Since $R$ is Gorenstein $\sHomR (M, N) \cong \Ext ^1_R (M, \syz _R ^1(N))$ (cf. \cite[(12.10)]{Y}). 
One has $I_n \cong \syz _R ^1(I_n)$. 
Thus it is enough to compute the dimension of $\Ext ^1_R (-, I_1)$. 
Notice that that $I_1=(x,y)R$. 
We have the complex;
$$
\begin{CD}
\cdots @>>> {I_1}^{\oplus 2} @>{\beta = \left( \begin{smallmatrix}x & 0 \\ y^n & -x \end{smallmatrix}\right)}>> {I_1}^{\oplus 2} @>{\alpha = \left( \begin{smallmatrix}x & 0 \\ y^n & -x \end{smallmatrix}\right)}>>{I_1}^{\oplus 2}  @>>> \cdots ,  
\end{CD}
$$ 
and $\Ext _R ^1 (I_n, I_1) \cong \ker \alpha / \Im \beta$. 
Let $\left( \begin{smallmatrix}a \\ b \end{smallmatrix}\right)\in  {I_1}^{\oplus 2}$. 
Assume that $\alpha (\left( \begin{smallmatrix}a \\ b \end{smallmatrix}\right)) =  \left( \begin{smallmatrix}x & 0 \\ y^n & -x \end{smallmatrix}\right) \left( \begin{smallmatrix}a \\ b \end{smallmatrix}\right) =  \left( \begin{smallmatrix}ax \\ a y^n -bx \end{smallmatrix}\right) = 0$. 
Since $ax \in (x^2)$, $a \in (x) \cap I_1 = (x)$.  
Moreover, since $a y^n -bx \in (x^2)$, one has $a y^n -bx=cx^2$ for some $c \in k[\![x, y]\!]$. 
It implies that $b = a' y^n-cx \in (x, y^n)$ with $a=a'x$. 
Thus, $\Ext _R ^1 (I_n, I_1) \cong \left\{ \left( \begin{smallmatrix}a'x \\ -cx + a'y^n \end{smallmatrix}\right)\ \verb+|+ \  a'x \in (x)/(x^2 , xy), -cx + a'y^n\in (x, y^n)/(x^2, xy, y^{n+1})\right\} \cong k ^{\oplus 2}$. 
Hence $\dim _k \sHomR (I_n, I_1) = \dim_k \Ext_R^1 (I_n, I_1) = 2$. 
Similarly, we have $\Ext_R^1 (I, I_1) \cong (x)/(x^2, xy) \cong k$, so that $\dim _k \sHomR (I, I_1)=1$.

Next we attempt to compute $\dim _k \sHomR (-, I_n)$ by using Proposition \ref{B9}. 
For each $n \geq1$, there exists an AR-sequence: $0 \to I_n \to I_{n+1} \oplus I_{n-1} \to I_n \to 0$, where $I_0 \cong R$.  
We note again that every MCM $R$-module $M$ is isomorphic to $\syz _R ^1 (M)$, namely $M \cong M[-1]$ in $\sCMR$. 
Then we claim that the following equation holds for $n \geq 1$:

\begin{claim} 
We have the following equation. 
\begin{equation}\label{B2-1}
\dim _k \sHomR (-, I_{n})= n\dim _k \sHomR (-, I_{1})-\sum _{i=1} ^{n-1} 2(n-i)\underline{\mu} (-, I_i). 
\end{equation}
\end{claim}

We prove it by induction on $n$. 
The case $n=1$ is clear. 
For $n >1$, by Proposition \ref{B9}, we have  
$$
\dim _k \sHomR (-, I_{n+1}) = 2 \dim _k \sHomR(-, I_{n}) - \dim _k \sHomR(-, I_{n-1}) - 2 \underline{\mu} (-, I_{n}).
$$
By the induction hypothesis,
$$
\begin{array}{ll}
\dim _k \sHomR (-, I_{n+1}) &= 2 \{ n\dim _k \sHomR (-, I_{1})-\sum _{i=1} ^{n-1} 2(n-i)\underline{\mu} (-, I_i) \} \\
&- \{ (n-1)\dim _k \sHomR (-, I_{1})-\sum _{i=1} ^{n-2} 2(n-1-i)\underline{\mu} (-, I_i)\} \\
& - 2 \underline{\mu} (-, I_{n}) \\
&= \{ 2n-(n-1)\} \dim _k \sHomR (-, I_{1}) \\ 
&+ \sum _{i=1}^{n-2} 2 \{ 2(n-i)-(n-1-i)\}\underline{\mu} (-, I_i)\\
& -4 \underline{\mu} (-, I_{n-1}) - 2 \underline{\mu} (-, I_n) \\
&= (n+1) \dim _k \sHomR (-, I_{1}) + \sum _{i=1}^{n} 2 (n +1-i )\underline{\mu} (-, I_i). 
\end{array}
$$
Hence the equation (\ref{B2-1}) holds.  
\qed

Now we calculate the dimension of $\sHomR (I_m, I_n)$. 
Suppose that $m \geq n$.
By virtue of the equation (\ref{B2-1}), $\dim _k \sHomR (I_m, I_n) = n\dim _k \sHomR (I_m, I_{1})-\sum _{i=1} ^{n-1} 2(n-i)\underline{\mu} (I_m, I_i) = n\dim _k \sHomR (I_m, I_{1})=2n$. 
If $m \leq n$ then we have the equality
$$
\begin{array}{ll}
\dim _k \sHomR (I_m, I_{n})&= n\dim _k \sHomR (I_m, I_{1})-\sum _{i=1} ^{n-1} 2(n-i)\underline{\mu} (I_m, I_i)\\
&= 2n - 2(n-m) \underline{\mu} (I_m, I_m) = 2m. 
\end{array}
$$
Moreover, the equation (\ref{B2-1}) also shows that $\dim _k \sHomR (I, I_{n})= n\dim _k \sHomR (I, I_{1})-\sum _{i=1} ^{n-1} 2(n-i)\underline{\mu} (I, I_i)= n$. 
By AR duality (see Remark \ref{B9}), we also have  $\dim _k \sHomR (I_{n}, I)= n$. 
Therefore the assertions (1) and (2) hold. 
(3) follows from the isomorphism $\Ext _R ^1 (I, I) \cong (x)/(x^2) \cong k[\![y]\!]$.  
\end{proof}

\begin{remark}\label{B9}
It is known as AR duality that $\HomR (\sHomR (M, N), E(k)) \cong \Ext ^1 (N, \tau M)$ for $N \in \CMR _0$. 
Here $E(k)$ is the injective hull of $k$ and $\tau M$ is the AR translation of $M$. See \cite[(3.10)]{Y}. 
It follows from the Matlis duality that $\dim_k \sHomR  (M, N) = \dim _k \sHomR  (N, \tau M)$. 
Hence, in Lemma \ref{B2} (1)(2), one can show from either equation the other equation. 
For instance, one has $\dim _k \sHomR (-, I_{n}) = \dim _k \sHomR (I_n,- )$ since $\tau M \cong M$ in this case. 
\end{remark}

\begin{lemma}\label{B3}
For $n>0$, we have the exact sequence: 
\begin{equation}\label{sequence}
0\to I_{n} \xrightarrow{\left( \begin{smallmatrix}\frac{x}{y^{n}} \\ 1\end{smallmatrix}\right)} I \oplus R \xrightarrow{(y^{n} \ -x)} I \to 0. 
\end{equation}
\end{lemma}

\begin{proof}
It is straightforward. 
\end{proof}

Thanks to Lemma \ref{B3}, we obtain the finitely presented functor:
\begin{equation}\label{B3-1}
0 \to \Hom_R(-, I_{n}) \to \Hom _R (-, I) \oplus \Hom_R (-, R) \to \Hom_R (-, I) \to H_n \to 0.  
\end{equation}
First, we shall show the functor $H_1$ is a simple object in $\smodC/ \smodC_0$.

\begin{proposition}\label{B4}
The functor $H_1$ is simple in $\smodC/\smodC_0$. 
\end{proposition}

\begin{proof}
By \cite[Proposition 3.3]{Y05}, the exact sequence (\ref{B3-1}) induces the long exact sequence:
\begin{equation}\label{long sequence}
\begin{CD}
@>>> H_1 @>>> 0@.@. \\
@>>> \sHomR(-, I_1) @>{\frac{x}{y}}>> \sHomR(-, I) @>{y}>> \sHomR(-, I) \\
@>>> \sHomR(-, I_1[-1]) @>>> \sHomR(-, I[-1]) @>>> \sHomR(-, I[-1]) . 
\end{CD}
\end{equation}
For each indecomposable $X \in \CM_0 (R)$, since $\dim _k \sHomR (X, I_1)$ and $\dim _k \sHomR (X, I)$ is finite, we have $\dim _k H_1 (X) = \frac{1}{2} \dim _k \sHomR (X, I_1)=1$. 
Since $\sHomR(I,I) \cong k[\![y]\!]$, one has $H_1 (I) \cong k[\![y]\!]/yk[\![y]\!]$. 
Consequently $\dim _k H_1 (X) = 1$ for all indecomposable $X \in \CMR$.

Let $0 \to V \to H_1 \to C \to 0$ be an admissible exact sequence in $\smodC$. 
Since $V \in \smodC$, we have the exact sequence $0 \to \Hom_R(-, Z) \to \Hom _R (-, Y) \to \Hom_R(-, X) \to V \to 0$. 
Then, for all $M \in \CM_0 (R)$,  
$$
\dim _k V(M) = \frac{1}{2} \left\{ \dim _k \sHomR (M, X) + \dim _k \sHomR (M, Z) -\dim _k \sHomR (M, Y)  \right\}. 
$$
Let $X = I ^{\oplus a_0} \oplus I_{l_1}^{\oplus a_1} \oplus \cdots \oplus  I_{l_{l'}}^{\oplus a_{l'}}$, $Y = I ^{\oplus b_0} \oplus I_{m_1}^{\oplus b_1} \oplus \cdots \oplus  I_{m_{m'}}^{\oplus b_{m'}}$ and $Z = I ^{\oplus c_0} \oplus I_{n_1}^{\oplus c_1} \oplus \cdots \oplus  I_{n_{n'}}^{\oplus c_{n'}}$. 
We put $m= max \{ l_1, ..., l_{l'}, m_1 , ..., m_{m'}, n_1, ..., n_{n'} \}$. 
For $m \leqq n < \infty$, 
$$
\dim _k V(I_n ) = \frac{1}{2}\left(  \sum_i ^{l'} m \cdot a_i + \sum_i ^{n'} m \cdot c_i  - \sum_i ^{m'} m \cdot b_i\right). 
$$
This equation yields that $\dim _k V(I_n )$ are $0$ or $1$ for $m \leqq n < \infty$ since $V$ is a subfunctor of $H_1$. 
Assume that $\dim _k V(I_n ) = 0$ for $m \leqq n$. 
Then $V(I_n) =0$ except for a finite number of $I_n$. 
Namely $\Supp (V)$ is a finite set, and we shall show $I \not \in \Supp (V)$. 
If it holds, $V$ is  in $\smodC _0$. 
Assume that $I \in \Supp (V)$. 
For $I' \in \Supp (V)\bigcap \CM_0 (R)$, there is an epimorphism from $V \to S_{I'}$. (See the proof of \cite[(4.12)]{Y}.) 
Put the kernel of the epimorphism as $V'$. 
Then $V' \in \smodC$ and $\Supp (V') = \Supp (V) \backslash \{ I' \}$. 
Repeating the procedure, we obtain the functor $\tilde{V}\in \smodC$ such that $\Supp (\tilde{V}) = \{ I \}$ and $\dim _k \tilde{V}(I) = 1$. 
It yields that $\tilde{V}$ is a simple object with $\tilde{V}(I) \not= 0$, so that the AR sequence ending in $I$ exists (\cite[(4.13)]{Y}). 
Namely $I \in \CM_0 (R)$ (\cite[(3.4)]{Y}). 
This is a contradiction. 
Hence $I \not \in \Supp (V)$. 

Assume that $\dim _k V(I_n ) = 1$ for $m \leqq n$. 
Then $\dim _k C(I_n) = 0$ for $m \leqq n$.  
Apply the same argument to $C$ and we also conclude that $C$ is contained in $\smodC _0$. 
Consequently we get the assertion.  
\end{proof}

\begin{remark}\label{B5}
Since $H_1$ is a subfunctor of $\sHomR (-, I_1)$, we have an exact sequence $0 \to H_1 \to \sHomR (-, I_1) \to H' _1 \to 0$ in $\smodC$. 
By the calculation in the proof of Proposition \ref{B4}, $\dim_k H' _1(I_n)=1$ for all $n$ and $\dim_k H' _1(I )=0$. 
By using the same argument of  Proposition \ref{B4}, one can also show that $H' _1$ is a simple object in $\smodC/\smodC_0$. 
Therefore, $\ell (\sHomR (-, I_1) )=2$ in $\smodC/\smodC_0$.  
\end{remark}

\begin{proposition}\label{B6}
The length of $\sHomR (-, I_n)$ is finite in $\smodC/\smodC_0$ for $n\geq 1$. 
\end{proposition}

\begin{proof}
First we claim that $\ell (\sHomR (-, I_1) ) < \infty$ in $\smodC/\smodC_0$, and it holds by Proposition \ref{B4} and Remark \ref{B5}. 
Suppose that $n>1$. 
Since there is an AR sequence $0 \to I_{n} \to I_{n+1}\oplus I_{n-1} \to I_{n} \to 0$ (\cite[6.1]{S87}), we obtain the sequence: 
$
\sHom_R(-, I_{n}) \to \sHom_R(-, I_{n+1})\oplus \sHom_R(-, I_{n-1}) \to \sHom_R(-, I_{n}). 
$
Since $\smodC_1$ is a Serre subcategory, $\sHomR (-, I_n)$ belongs to $\smodC_1$. 
That is $\ell (\sHomR (-, I_n)) < \infty$ in $\smodC/\smodC_0$. 
\end{proof}

\begin{remark}
In the Grothendieck group of $\smodC$, an AR sequence gives $[\sHom_R(-, I_{n+1})]+ [\sHom_R(-, I_{n-1})] = 2 [\sHom_R(-, I_{n})] - 2 [S_{I_n}]$. 
Combing the equality with Remark \ref{B5}, one has $\ell (\sHomR (-, I_n))= 2n$ in $\smodC/\smodC_0$ for $n \geq 1$. 
By \cite[Proposition 2.1 (1)]{AIT12}, there is an exact sequence $0 \to I \to I_n \to I \to 0$ for $n \geq 1$. 
Then $2\ell (\sHomR (-, I)) \geq \ell (\sHomR (-, I_n ))$ in $\smodC/\smodC_0$. 
This yields that $\sHomR (-, I)$ does not belong to $\smodC_1$. 
\end{remark}

\begin{lemma}\label{B7}
Let $F$ be a finitely presented functor with the exact sequence
$
0 \to \Hom _R (-, Z) \to \Hom _R (-, Y) \to \Hom _R (-, X \oplus C) \to F \to 0. 
$
Then there is an exact sequence of functors
$$
F' \xrightarrow{\rho} F \to F/\Im \rho \to 0
$$
such that $F'$ and $F/\Im \rho$ are images of $\Hom _R(-, X)$ and $\Hom _R(-, C)$ respectively. 
\end{lemma}

\begin{proof}
The following diagram is obtained by the taking pullback:  
$$
\begin{CD}
@. @.0 @.0 @.\\
@.@.  @AAA  @AAA \\
@.@.C  @= C \\
@.@.  @AAA  @AAA \\
0 @>>> Z @>>> Y @>>> X\oplus C @>>>0\\
@. @| @AAA  @AAA \\
0 @>>> Z @>>> P @>>> X @>>>0\\
@.@.  @AAA  @AAA \\
@. @.0 @.0. @.\\
\end{CD}
$$
Thus we get 
$$
\begin{CD}
@. @. @.0 @.0 @.\\
@.@.  @.  @AAA @AAA\\
@.@. @. \Hom _R(-, C) @>>> F/\Im \rho @>>>0 \\
@.@. @.  @AAA @AAA \\
0 @>>> \Hom _R(-, Z) @>>> \Hom _R(-, Y) @>>> \Hom _R(-, X\oplus C) @>>> F@>>>0\\
@. @| @AAA  @AAA @A{\rho}AA \\
0 @>>>\Hom _R(-, Z) @>>> \Hom _R(-, P) @>>> \Hom _R(-, X) @>>> F' @>>>0\\
@.@.  @AAA  @AAA \\
@. @.0 @.0. @.\\
\end{CD}
$$
\end{proof}

\begin{proposition}\label{B8}
The functor $\sHom _R (-, I)$ is simple in $\smodC/\smodC_1$. 
\end{proposition}

\begin{proof}
Let
\begin{equation}\label{sequence of I}
0 \to V \to \sHomR (-, I) \to C \to 0
\end{equation}
be an admissible exact sequence in $\smodC$. 
Since $V \in \smodC$, we may assume that $V$ has the sequence $0 \to \Hom_R (-, Z)\to \Hom_R (-, Y)\to \Hom_R (-, X)\to V \to 0$ for some $X, Y, Z \in \CMR$. 
If $X \in \CM_0 (R)$, $V \in \smodC_1 $ because $V$ is an image of $\sHomR (-, X)$ (\cite[(4.16)]{Y}). 
Thus the proof is completed. 
Therefore we assume that $X$ contains $I$ as a direct summand. 
Let $X \cong I^{\oplus l} \oplus M$ where $M \in \CM_0 (R)$. 
By Lemma \ref{B7}, there exits the sequence $V' \xrightarrow{\rho} V \to V/\Im \rho \to 0$ such that $V'$ and $V/\Im \rho$ are images of $\Hom_R(-, I^{\oplus l})$ and $\Hom_R (-, M)$ respectively. 
Then the commutative diagram
$$
\begin{CD}
@.0@.@.\\
@.@AAA @.@.\\
@.V/\Im \rho @.0 @.0 \\
@.@AAA @AAA@AAA\\
0@>>> V @>>> \sHomR (, I)@>>> C @>>>0\\
@.@AAA @| @AAA @. \\
0@>>> \Im \rho @>>> \sHomR (, I)@>>> C' @>>>0\\
@.@AAA @AAA @AAA @.\\
@.0@.0@.V/\Im \rho @.\\
@.@.@. @AAA @.\\
@.@.@.0@.\\
\end{CD}
$$
is obtained. 
Since $V/\Im \rho \in \smodC_0$, $V \cong \Im \rho$ and $C \cong C'$ in $\smodC/ \smodC_0$. 
Hence we may consider $0 \to \Im \rho \to \sHomR (-, I) \to C' \to 0$ instead of the sequence (\ref{sequence of I}). 
Now we remark that $\Im \rho$ is also the image of $\Hom_R(-, I^{\oplus l})$, so that we have 
$$
\sHomR (- , I^{\oplus l}) \xrightarrow{\varphi}  \sHomR (-, I) \to C' \to 0. 
$$
By Yoneda's lemma, $\varphi \in \sHomR (I^{\oplus l}, I)$. 
Since $\sHomR (I, I)\cong k[\![y]\!]$, $\varphi$ is of the form $\left(  y^{n_1} \  y^{n_2} \  \cdots \ y^{n_l}\right)$, where $n_1 \leq n_2 \leq \cdots \leq n_l$. 
The mapping cone of $C(\varphi )$ is $I_{n_1} \oplus I^{\oplus l-1}$. 
Actually, 
$$
C(\varphi ) = \Coker \left( \begin{smallmatrix}x &&&y^{n_1}\\  &\ddots& &\vdots \\  &&x&y^{n_l}\\  &&&x\\ \end{smallmatrix} \right): R^{\oplus l+1} \to R^{\oplus l+1}. 
$$
By considering basic matrix transformations, one can obtain the isomorphism $C(\varphi ) \cong I_{n_1} \oplus I^{\oplus l-1}$. 
Since $C(\varphi )[-1] \cong (I_{n_1} \oplus I^{\oplus l-1})[-1] \cong I_{n_1} \oplus I^{\oplus l-1}$, we can make the triangle: 
$$
(I_{n_1} \oplus I^{\oplus l-1}) \to I^{\oplus l} \xrightarrow{\varphi} I \xrightarrow{\psi} (I_{n_1} \oplus I^{\oplus l-1})[1]. 
$$
Now we claim that $I^{\oplus l-1}$ is split out. 
Set $\psi = \left( \begin{smallmatrix} \psi_1 \\ \psi_2 \\ \vdots \\ \psi_l \end{smallmatrix}
\right)$, and we shall show $\psi _i = 0$ for $i \geq 2$. 
Since $\psi \circ \varphi=0$, 
$$
\left( \begin{smallmatrix}\psi_1 y^{n_1}& \psi_1y^{n_2}  &\cdots &\psi_1 y^{n_l}\\ 
\psi_2 y^{n_1}& \psi_2y^{n_2}  &\cdots &\psi_2 y^{n_l}\\
\vdots&&\ddots&\vdots\\ \psi_l y^{n_1}& \psi_ly^{n_2}  &\cdots &\psi_l y^{n_l}\\\end{smallmatrix}\right) = 0. 
$$ 
Then, for $i \geq 2$, $\psi _i y^{n_i} = y^{n_i} \psi _i= 0$ in $ \sHomR (I, I)$. 
Since $\sHomR (I, I) \cong k[\![y]\!]$, $y^{n_i}$ is a non zero divizor on $ \sHomR (I, I)$, so that $\psi _i = 0$. 
Consequently, $\psi _i = 0$ for all $i \geq 2$. 
Thus we have the diagram:
$$
\begin{CD}
I @>{\psi _1}>> I_{n_1}  @>{\alpha '}>> C(\psi_1) @>>> I\\
@| @VV{\left( \begin{smallmatrix}1\\0\end{smallmatrix}\right)}V @VV{\gamma}V @|\\
I@>{\left( \begin{smallmatrix}\psi_1\\0\end{smallmatrix}\right)}>> I_{n_1} \oplus I^{\oplus l-1} @>{\alpha}>> I^{\oplus l} @>>> I\\ 
@VVV @VV{\left( \begin{smallmatrix}0&1\end{smallmatrix}\right)}V @VV{\beta}V @VVV\\
0 @>>> I^{\oplus l-1}@= I^{\oplus l-1} @>>>0. 
\end{CD}
$$
One can see that $\beta$ is a retraction. 
Thus $\gamma$ is a section. 
Moreover $C(\psi_1 ) \cong I$. 
Therefore we get
$$
\begin{CD}
I@>{\left( \begin{smallmatrix}\psi_1\\0\end{smallmatrix}\right)}>> I_{n_1} \oplus I^{\oplus l-1} @>{\left( \begin{smallmatrix}\alpha' &0 \\ 0&1\end{smallmatrix}\right)}>> I^{\oplus l} @>>> I\\ 
@| @| @VV{\left( \begin{smallmatrix}\gamma & \beta \end{smallmatrix}\right)}V @|\\
I@>{\left( \begin{smallmatrix}\psi_1\\0\end{smallmatrix}\right)}>> I_{n_1} \oplus I^{\oplus l-1} @>{\alpha}>> I^{\oplus l} @>>> I. \\
 \end{CD}
$$
Since $\beta$ (resp. $\gamma$) is a retraction (resp. a section), $\left( \begin{smallmatrix}\gamma & \beta \end{smallmatrix}\right)$ gives an isomorphism.  
Therefore we may assume that $C'$ has the presentation:
$$
\sHomR(-, I_{n_1})\to \sHomR(-, I) \to \sHomR(-, I) \to C' \to 0. 
$$
It implies that $C'$ is a subfunctor of $\sHomR(-, I_{n_1}[1]) \cong \sHomR(-, I_{n_1})$. 
By Proposition \ref{B6} the length of $C'$, hence $C$, has finite length in $\smodC/\smodC_0$.
This observation shows  that $\sHomR ( - , I)$ is simple in $\smodC/\smodC_1$. 
\end{proof}

\begin{proof}[Proof of Theorem \ref{B1}.]
For each $F \in \smodC$, we have an epimorphism $\Hom_R(-, X) \to F \to 0$. 
In particular, the epimorphism
$$
\sHom_R(-, X) \to F \to 0
$$
exists, where $X \in \CMR$. 
From the former proposition, $\ell (\sHomR (-, X)) < \infty$ in $\smodC / \smodC_1$. 
Hence $\ell(F ) < \infty$ in $\smodC / \smodC_1$. 
This shows that $\KGdim \smodC = 2$. 
\end{proof}

%%%%%%%%%%%%%%%%%%%%%%%%%%%%% 
%%%%%%%%%%%%%%%%%%%%%%%%%%%%% 
%%%%%%%%%%%%%%%%%%%%%%%%%%%%% 
\section{Krull--Gabriel dimension of $\smod (\CM (k[\![x,y,z]\!]/(x^2y+z^2)))$}\label{D}

In this section, we investigate the Krull--Gabriel dimension over a hypersurface of type $(D_{\infty})$. 
We shall calculate it when the hypersurface is of 2-dimensional, namely,
$$
k[\![x, y, z]\!]/(x^2y+z^2). 
$$
We propose to generalize the calculation for high (or low) dimensional cases in Section \ref{C}. 
As pointed out in \cite[Remark 5.10.]{BD08}, the stable categories $\sCM (k[\![x,y,z]\!]/(x^2y+z^2))$ and $\sCM (k[\![x,y]/(x^2))$ are similar. 
For example the object is isomorphic to its first syzygy module.  
Actually we can apply the arguments in Section \ref{B} to compute the Krull-Gabriel dimension of $\smod (\CM (k[\![x,y,z]\!]/(x^2y+z^2)) )$.

Let $\Rs = k[\![x,y,z]\!]/(x^2y+z^2)$. 
By \cite[Proposition 14.19]{LW12} (see also \cite[(5.7)]{BD08}), all non free indecomposable MCM $\Rs$-modules are
\begin{align*}
&I=\Coker  \left( \begin{smallmatrix} z & -xy \\ x & z\end{smallmatrix} \right) :{\Rs}^{\oplus 2} \to {\Rs}^{\oplus 2} ,\,
M_0 = \Coker \left( \begin{smallmatrix} z & -y \\ x^2 & z \end{smallmatrix} \right) :{\Rs}^{\oplus 2} \to {\Rs}^{\oplus 2}  ,
\\
&M_n =\Coker \left( \begin{smallmatrix} z&0&x&y^n\\0&z&0&-x\\xy&y^{n+1}&z&0\\0&-xy&0&z\\\end{smallmatrix}\right) : {\Rs} ^{\oplus 4} \to {\Rs} ^{\oplus 4} ,\,
\\
&N_n = \Coker \left( \begin{smallmatrix} z&0&xy&y^n\\0&z&0&-x\\x&y^n&z&0\\0&-xy&0&z\\\end{smallmatrix}\right) : {\Rs}^{\oplus 4} \to {\Rs}^{\oplus 4} \quad  (n\ge1).
\end{align*}
Note that $I \cong \syz _{\Rs} ^1 I$, $M_n \cong \syz _{\Rs} ^1 M_n$ and $N_n \cong \syz _{\Rs} ^1  N_n$. 
For convenience, we put
$$
L_n := \left\{ \begin{array}{l}
M_{\frac{n-1}{2}} \quad n : odd,\\
N_{\frac{n}{2}} \quad n : even.\\
\end{array}
\right.
$$

\begin{lemma}\label{D1}
Let $\Rs$, $I$, $L_n$ be as above. 
Then the following statements hold.  
\begin{enumerate}[\rm(1)]
\item $\dim _k \sHomRs (L_m, L_n) = \left\{ \begin{array}{l} 2n  \quad m \geq n, \\2m \quad m \leq n. \end{array}\right.$
\item $\dim _k \sHomRs (I, L_n) = \dim _k \sHomR (L_n, I) = m $ for $n \geq 1$. 
\item $\dim _k \sHomRs (I, I) = \infty$. 
\end{enumerate}
\end{lemma}

\begin{proof}
We compute them by the same methods used in Lemm \ref{B2}. 
Since there are AR triangles $L_{n} \to L_{n+1} \oplus L_{n-1} \to L_{n} \to L_{n}[1]$ for $n \geq1$, here $L_0 = 0$, the equation 
\begin{equation}\label{D1-1}
\dim _k \sHomRs (-, L_{n})= n\dim _k \sHomRs (-, L_{1})-\sum _{i=1} ^{n-1} 2(n-i)\underline{\mu} (-, L_i) 
\end{equation}
holds. 
Hence it is enough to compute the dimension of $\sHomRs (-, L_{1})$. 
To do this, we use the results due to Kn\"orrer (see Section \ref{C} for the detail). 
Let $R=k[\![x,y]\!]/(x^2y)$. 
All non free indecomposable MCM $R$-modules are
\begin{align*}
&R/(x),\,
R/(xy),\,
R/(x^2),\,
R/(y),  \\
&M_n ^+ =\Coker \left( \begin{smallmatrix} x&y^n\\0&-x\\\end{smallmatrix}\right) : {R} ^{\oplus 2} \to {R} ^{\oplus 2} ,\, M_n ^- =\Coker \left( \begin{smallmatrix} xy&y^{n+1}\\0&-xy\\ \end{smallmatrix}\right) : {R} ^{\oplus 2} \to {R} ^{\oplus 2} ,\\
&N_n ^+ =\Coker \left( \begin{smallmatrix} xy&y^n\\0&-x\\\end{smallmatrix}\right) : {R} ^{\oplus 2} \to {R} ^{\oplus 2} ,\, N_n ^- =\Coker \left( \begin{smallmatrix} x&y^{n}\\0&-xy\\ \end{smallmatrix}\right) : {R} ^{\oplus 2} \to {R} ^{\oplus 2}, (n\ge1).
\end{align*}
As mentioned in Corollary \ref{C12}, there is an adjoint pair of functors $(\AA, \BB)$ between $\sCMR$ and $\sCMRs$. 
Particularly we have the isomorphism $\sHom_{\Rs} (\AA (-), L_1) \cong \sHom _R (-, \BB (L_1))$. 
Notice that $\AA (R/(x)) = \AA (R/(xy)) = I$, $\AA (R/(x^2)) = \AA (R/(y)) = L_1$, $\AA (M_n ^+) = \AA (M_n ^- ) = L_{2n+1}$ and $\AA (N_n ^+) = \AA (N_n ^- ) = L_{2n}$. 
Hence we may compute the dimension of $\sHom _R (-, \BB (L_1))$ instead of $\sHom_{\Rs} (-, L_1)$. 
Since $\BB (L_1) = R/(x^2) \oplus R/(y)$, we have $\dim_k \sHom _R (-, \BB (L_1)) = \dim_k \sHom _R (-, R/(x^2) \oplus R/(y))$. 

\begin{claim}
We have the following equality. 
\begin{enumerate}[\rm(i)]
\item $\dim_k \sHom _R (R/ (y) , R/(x^2) \oplus R/(y)) = 2$.
\item $\dim_k \sHom _R (M_n ^+ , R/(x^2) \oplus R/(y)) = 2$. 
\item $\dim_k \sHom _R (N_n ^+ , R/(x^2) \oplus R/(y)) = 2$. 
\item $\dim_k \sHom _R (R/(x) , R/(x^2) \oplus R/(y)) = 1$. 
\item $\dim_k \sHom _R (R/(x) , R/(x) \oplus R/(xy)) = \infty$. 
\end{enumerate}
\end{claim}

\begin{proof}
(i) First we note that $R/(x^2) \cong (y)$ and $R/(y) \cong (x^2)$. 
It follows from the complexes
$$
\begin{CD}
\cdots @>>> (y) @>{y}>> (y) @>{x^2}>>(y)  @>>> \cdots 
\end{CD}
$$
and 
$$
\begin{CD}
\cdots @>>> (x^2) @>{y}>> (x^2) @>{x^2}>>(x^2)  @>>> \cdots 
\end{CD}
$$
that $\Ext _R ^1 (R/(y), (y)) \cong (y)/(y^2) \cong k^{\oplus 2}$ and $\Ext _R ^1 (R/(y), (x^2)) = (0)$. 
Thus $\dim_k \sHom _R (R/ (y) , R/(x^2) \oplus R/(y)) = \dim _k \Ext _R ^1 (R/(y), (y)) = 2$.

(ii) We have the complex:
$$
\begin{CD}
\cdots @>>> {(x^2)}^{\oplus 2} @>{\beta = \left( \begin{smallmatrix}x & 0 \\ y^n & -x \end{smallmatrix}\right)}>> {(x^2)}^{\oplus 2} @>{\alpha = \left( \begin{smallmatrix}xy & 0 \\ y^{n+1} & -xy \end{smallmatrix}\right)}>>{(x^2)}^{\oplus 2}  @>>> \cdots. 
\end{CD}
$$ 
Suppose that $\alpha (\left( \begin{smallmatrix} a \\ b \end{smallmatrix}\right)) = \left( \begin{smallmatrix} axy \\ ay^{n+1}-bxy \end{smallmatrix}\right) = 0$. 
Then $axy \in (x^2y)$, so that $a \in (x^2) \cap (x) = (x^2)$. 
Since $ay^{n+1}-bxy \in (x^2y)$ and $a \in (x^2)$, $b \in (x^2) \cap (x) = (x^2)$. 
Thus $\ker \alpha \cong (x^2) \oplus (x^2)$. 
Since $\beta (\left( \begin{smallmatrix} a \\ b \end{smallmatrix}\right))= \left( \begin{smallmatrix} ax \\ ay^{n}-bx \end{smallmatrix}\right)$, $\Im \beta \cong (x^3) \oplus (x^3)$.  
We have that  $\Ext _R ^1 (M_n ^+, (x^2)) = \ker \alpha / \Im \beta \cong (x^2)/(x^3) \oplus (x^2)/(x^3)\cong k^{\oplus 2}$.  
Also, we have the complex: 
$$
\begin{CD}
\cdots @>>> {(y)}^{\oplus 2} @>{\beta = \left( \begin{smallmatrix}x & 0 \\ y^n & -x \end{smallmatrix}\right)}>> {(y)}^{\oplus 2} @>{\alpha = \left( \begin{smallmatrix}xy & 0 \\ y^{n+1} & -xy \end{smallmatrix}\right)}>>{(y)}^{\oplus 2}  @>>> \cdots.  
\end{CD}
$$
Let $\left( \begin{smallmatrix} a \\ b \end{smallmatrix}\right) \in (y)^{\oplus 2}$ be an element such that $\alpha (\left( \begin{smallmatrix} a \\ b \end{smallmatrix}\right))=0$. 
Then $a$ (resp. $b$) is in  $(y) \cap (x) = (xy)$ (resp. $(y) \cap (x, y^{n+1}) = (xy, y^{n+1})$). 
Since $\Im \beta = (xy) \oplus (xy, y^{n+1})$, $\Ext _R ^1 (M_n ^+, (y)) \cong (xy)/(xy) \oplus (xy, y^{n+1})/(xy, y^{n+1}) = (0)$.
Therefore $\dim_k \sHom _R (M_n ^+ , R/(x^2) \oplus R/(y)) = \dim_k \Ext _R ^1 (M_n ^+ , R/(x^2) \oplus R/(y)) = 2$. 

(iii) Similarly, the complexes
$$
\begin{CD}
\cdots @>>> {(x^2)}^{\oplus 2} @>{\left( \begin{smallmatrix}xy & 0 \\ y^n & -x \end{smallmatrix}\right)}>> {(x^2)}^{\oplus 2} @>{\left( \begin{smallmatrix}x & 0 \\ y^{n+1} & -xy \end{smallmatrix}\right)}>>{(x^2)}^{\oplus 2}  @>>> \cdots 
\end{CD}
$$
and 
$$
\begin{CD}
\cdots @>>> {(y)}^{\oplus 2} @>{\left( \begin{smallmatrix}xy & 0 \\ y^n & -x \end{smallmatrix}\right)}>> {(y)}^{\oplus 2} @>{\left( \begin{smallmatrix}x & 0 \\ y^{n+1} & -xy \end{smallmatrix}\right)}>>{(y)}^{\oplus 2}  @>>> \cdots 
\end{CD}
$$
say that $\Ext _R ^1 (N_n ^+ , (x^2)) \cong (x^2)/(x^3) \cong k$ and $\Ext _R ^1 (N_n ^+ , (y)) \cong (xy, y^n)/(xy, y^{n+1}) \cong k$, so that $\dim_k \sHom _R (N_n ^+ , R/(x^2) \oplus R/(y)) = \dim_k \Ext _R ^1 (N_n ^+, (x^2) \oplus (y))= 2$.

(iv) From the complexes
$$
\begin{CD}
\cdots @>>> (y) @>{x}>> (y) @>{xy}>>(y)  @>>> \cdots 
\end{CD}
$$
and 
$$
\begin{CD}
\cdots @>>> (x^2) @>{x}>> (x^2) @>{xy}>>(x^2)  @>>> \cdots 
\end{CD}
$$
one has $\Ext _R ^1 (R/(x), (y)) =(0)$ and $\Ext _R ^1 (R/(x), (x^2)) \cong (x^2)/(x^3) \cong k$. 
Therefore $\dim_k \sHom _R (R/(x) , R/(x^2) \oplus R/(y)) = \dim _k \Ext _R ^1 (R/(x), (x^2)) = 1$. 

(v) Notice that $R/(x) \cong (xy)$ and $R/(xy) \cong (x)$. 
One can show from the complexes
$$
\begin{CD}
\cdots @>>> (xy) @>{x}>> (xy) @>{xy}>>(xy)  @>>> \cdots 
\end{CD}
$$
and 
$$
\begin{CD}
\cdots @>>> (x) @>{x}>> (x) @>{xy}>>(x)  @>>> \cdots 
\end{CD}
$$
that $\Ext _R ^1 (R/(x), (xy)) \cong (xy)$ and $\Ext _R ^1 (R/(x), (x)) \cong (x)/(x^2)$. 
This implies (v). 
\end{proof}

The claim (i), (ii), (iii), (iv) and the equation (\ref{D1-1}) give the assertion (1), (2). 
The claim (v) also gives (3). 
\end{proof}

\begin{remark}\label{D2}
As mentioned in the proof of Lemma \ref{D1}(3), $\sHomRs (I, I)$ is isomorphic to $k[\![y]\!] \oplus k[\![y]\!]$. 
Let $T = S[\![z]\!]/(f+z^2 )$ with $f \in \m _S$. 
Then there is a one-to-one correspondence between the isomorphism classes of MCM $T$-modules and the equivalence classes of square matrices $\varphi$ with entries in $S$ such that $\varphi ^2 = -f$.
We also remark that giving a $T$-homomorphism $g : M \to N$ is equivalent to giving an $S$-homomorphism $\alpha : S^{\oplus m} \to S^{\oplus n}$ such that $\alpha \circ \varphi _M = \varphi _N \circ \alpha$. 
By using this, one has $\Hom_{\Rs} (I, I) \cong \left\{ \left( \begin{smallmatrix}a&-cy \\ c &a \end{smallmatrix}\right) \verb+|+ a, c \in k[\![x, y]\!] \right\}$. 
Note that  $I$ correspondences to the matrix $\left( \begin{smallmatrix}0&-xy \\ x &0\end{smallmatrix}\right)$. 
Moreover, since $\Rs$ correspondences to $\left( \begin{smallmatrix}0&-x^2y \\ 1 &0 \end{smallmatrix}\right)$, $P(I, I)$ is isomorphic to $\left\{ \left( \begin{smallmatrix}ax&-cxy \\ cx &ax \end{smallmatrix}\right) \verb+|+ a, c \in k[\![x, y]\!] \right\}$, so that  we have $\sHomRs (I, I) \cong \left\{ \left( \begin{smallmatrix}a&-cy \\ c &a \end{smallmatrix}\right) \verb+|+ a, c \in k[\![y]\!] \right\}$. 
And we see that $\frac{z}{x} y^n$ (resp. $y^n$) in $\sHomRs (I, I)$ correspondences to $\left( \begin{smallmatrix}0&-y^{n+1} \\ y^n &0 \end{smallmatrix}\right)$ (resp, $\left( \begin{smallmatrix}y^n&0 \\ 0 &y^n \end{smallmatrix}\right)$). 
Therefore $\sHomRs (I, I)$ is generated by $\frac{z}{x}$ and $1$ as a $k[\![y]\!]$-module. 
Now we calculate the mapping cone $C(\varphi) : I \to I$. 
One has 
\begin{align*}
& C(\frac{z}{x} y^n )  = \Coker \left( \begin{smallmatrix} z&-xy&0&y^{n+1}\\x&z&-y^n&0\\0&0&z&xy\\0&0&-x&z\\\end{smallmatrix}\right) :{\Rs } ^{\oplus 4} \to {\Rs }^{\oplus 4} \cong M_n =L_{2n+1}, \\
& C(y^n) = \Coker \left( \begin{smallmatrix} z&-xy&-y^n&0\\x&z&0&-y^n\\0&0&z&xy\\0&0&-x&z\\\end{smallmatrix}\right) :{\Rs } ^{\oplus 4} \to {\Rs }^{\oplus 4} \cong N_n =L_{2n}
\end{align*}
respectively. 
\end{remark}

\begin{theorem}\label{D3}
Let $\Rs$, $I$, $L_n$ be as above. 
Then the following statements hold. 
\begin{enumerate}[\rm(1)]
\item $\KGdim \sHomRs (-, L_n) = 1$ for $n \geq 1$. 
\item $\KGdim \sHomRs (-, I) = 2$. 
\end{enumerate}
Consequently $\KGdim \smodCs = 2$. 
\end{theorem}

\begin{proof}
The arguments in Section \ref{B} are valid. 
We have the exact sequence: 
\begin{equation}\label{D3-1}
0\to L_{1} \xrightarrow{\left( \begin{smallmatrix}\frac{x^2}{z} \\ 1\end{smallmatrix}\right)} I \oplus R \xrightarrow{(\frac{z}{x} \ -x)} I \to 0. 
\end{equation}
We consider the functor induced by the sequence (\ref{D3-1}), that is, $\sHomRs (-, L_1) \to \sHomRs (-, I) \to \sHomRs (-, I) \to G_1 \to 0$. 
The functor $G_1$ satisfies the equation
$
\dim_k G_1 (L) =1
$ for each indecomposable $L \in \sCMRs$. 
Thus by using the same arguments in Proposition \ref{B4}, one has $G_1$ is a simple object in $\smodCs/\smodCs_0$. 
Moreover $\ell (\sHomRs (-, L_1 )) = 2$ (see Remark \ref{B5}). 
Since we have an AR-triangle $L_n \to L_{n+1} \oplus L_{n-1} \to L_n \to L_{n}[1]$ (\cite[6.2]{S87}), we have $\ell (\sHomRs (-, L_n )) = 2n$ in $\smodCs/\smodCs_0$, so that the assertion (1) holds. 
See also Proposition \ref{B6}.

According to \cite[Proposition 2.2. (2)]{AIT12}, there is a triangle $ I \to L_n \to I \to I[1]$ for any $n \geq 1$. 
Thus we have $2\ell (\sHomRs (-, I)) \geq \ell (\sHomRs (-, L_n ))$ in $\smodCs/\smodCs_0$. 
Hence $\sHomR (-, I)$ does not belong to $\smodCs_1$. 
Then we can also show $\KGdim \sHomRs (-, I) =2$ by using the arguments in Proposition \ref{B8}. 
Let $0 \to V \to \sHomRs (-, I) \to C \to 0$ be an admissible sequence in $\smodCs$. 
By virtue of Lemma \ref{B7}, we may assume that $C$ has the presentation $\sHomRs (-, I^{\oplus l}) \xrightarrow{\varphi} \sHomRs (-, I) \to C \to 0$. 
According to Remark \ref{D2}, $\varphi$ is of the form $\left( a_1 y^{n_1} + b_1 \dfrac{z}{x}y^{m_1}\  a_2 y^{n_2} + b_2 \dfrac{z}{x}y^{m_2}\  \cdots \ a_l y^{n_l} + b_l \dfrac{z}{x}y^{m_l} \right)$ where $a_i$, $b_i$ are units in $k[\![y]\!]$. 
Set $n= min \{ n_1, \cdots , n_l\}$ and $m= min \{ m_1, \cdots , m_l\}$. 
Suppose that $n \geq m$, the mapping cone $C(\varphi ) \cong M_m = L_{2m+1} \oplus I ^{\oplus l-1}$. 
Suppose that $n < m$, then $C(\varphi ) \cong N_n = L_{2n}\oplus I ^{\oplus l-1}$. 
Since $a_i y^{n_i} + b_i \dfrac{z}{x}y^{m_i}$ is a non zero divisor on $\sHomRs (I, I)$, $I ^{\oplus l-1}$ is split out as shown in Proposition \ref{B8}. 
Hence we also have $C$ is a subfunctor of $\sHomRs (-, L_{2m+1})$ or $\sHomRs (-, L_{2n})$, so that $C$ is  in $\smodCs_1$. 
Consequently, $\sHomRs (-, I)$ is simple in $\smodCs/\smodCs_1$, which means $\KGdim \sHomRs (-, I) = 2$. 
\end{proof}

%%%%%%%%%%%%%%%%%%%%%%%%%%%%% 
%%%%%%%%%%%%%%%%%%%%%%%%%%%%% 
%%%%%%%%%%%%%%%%%%%%%%%%%%%%% 
\section{ Kn\"orrer's periodicity}\label{C}

In this section we investigate how the Krull--Gabriel dimension behaves concerning Kn\"orrer's periodicity. 
We recall some observations given in \cite{RR85, Kn87}.

Let $\CC$ and $\DD$ be additive categories with a functor $\AA: \CC \to \DD$. 
Then $\AA$ induces the functor $\AA : \mod (\CC ) \to \mod (\DD )$ by $\AA ( \Hom _{\CC} (-, C)) = \Hom _{\DD} (-, \AA(C))$.

\begin{lemma}\label{C1}
Let $\CC$ and $\DD$ be additive categories with functors $\AA: \CC \to \DD$ and $\BB: \DD \to \CC$. 
Suppose that $(\BB, \AA)$ is an adjoint pair of functors. 
Then the induced functor $\AA : \mod(\CC ) \to \mod(\DD )$ is an exact functor. 
\end{lemma}

\begin{proof}
By the adjointness of $(\BB, \AA)$, one can show that $\AA (F)(-) \cong F (\BB (- ))$ for $F \in \smod (\CC )$. 
The assertion follows from the isomorphism. 
\end{proof}

Let $R$ be a hypersurface, that is, $R= S/(f)$ where $S = k[\![x_0, x_1, \cdots , x_n]\!]$ is a formal power series ring with a maximal ideal $\m _S = (x_0, x_1, \cdots , x_n )$ and $f \in \m _S$. 
For the ring $R$, we denote $\Rs = S[\![z]\!] /(f + z^2)$. 
Then the group $G=\Z /2\Z$ acts on $\Rs$ by $\sigma : z \to -z$. 
Denote the skew group ring by $\RsG$. 
We say that a finitely generated $\RsG$-module is MCM if it is MCM as an $\Rs$-module.  
For an $\Rs$-module $M$ and the involution $\sigma$ in $G$, we define an $\Rs$-module $\sigma^{\ast}  M$ by $M = \sigma ^{\ast} M$ as a set and $r \circ m = \sigma (r) m$. 
For the detail, refer to \cite[Section 2]{Kn87}.

\begin{remark}\label{C2}
For an $\Rs$-module $M$, $M \oplus \sigma \ast M$ has an $\RsG$-module structure. 
For $(a, b) \in M \oplus \sigma ^{\ast} M$, we define the action of $\sigma$ by $\sigma (a, b) = (b, a)$. 
Moreover, for an $\Rs$-homomorphism $f : M \to N$, we see that $\left( \begin{smallmatrix} f &0 \\ 0&f \end{smallmatrix}\right): M \oplus \sigma^{\ast}  M \to N \oplus \sigma^{\ast} N$ is an $\RsG$-homomorphism.   
\end{remark}

The following theorem is due to Kn\"orrer \cite{Kn87}. 

\begin{theorem}\label{C3}\cite{Kn87}
Let $R$, $\RsG$, $\Rs$ be as above. 
We have the functors
$$
\CMR \xrightarrow{\Omega}\CMRsG \stackrel[ad]{\mathcal{F}}{\rightleftarrows} \CMRs, 
$$
where the functor $\Omega(-)$ is defined by $\syz ^1 _{\Rs} (- )$, $\mathcal{F}$ is a forget-functor and $ad(-)=- \otimes_{\Rs} \RsG$ is its adjoint. 
Then, for $X \in \CMR$ and $Y \in \CMRs$, the following statements hold. 
\begin{enumerate}[\rm(1)]
\item The functor $\Omega$ gives the categorical equivalence. 
\item $\Omega^{-1} \circ ad \circ F \circ \Omega$ is equivalent to the functor $X \to X \oplus \syz _R ^1(X)$. 

\item $F \circ \Omega \circ \Omega^{-1} \circ ad$  is equivalent to the functor $Y \to  Y \oplus \sigma ^{\ast} Y$.

\end{enumerate}
\end{theorem}

\begin{proof}
(1) \cite[Proposition 2.1, Remark 2.2(ii)]{Kn87}. (2), (3) \cite[Proposition 2.4, Lemma 2.5]{Kn87}. 
\end{proof}

\begin{lemma}\label{C4}
Let $\Omega$, $\mathcal{F}$ and $ad$ be as above. 
Set $\AA = \mathcal{F} \circ \Omega$ and $\mathcal{B} = \Omega^{-1} \circ ad$. 
Then $(\AA , \mathcal{B})$ and $(\mathcal{B} , \AA)$ are adjoint pairs. 
\end{lemma}

\begin{proof}
Since $\Omega$ gives the equivalence it is enough to show that $(ad, \FF)$ and $(\FF , ad)$ are adjoint pairs. 
It is well-known that $(ad, \FF)$ is an adjoint pair, and we show $(\FF , ad)$ is an adjoint pair. 
Note that, for $Y \in \CMRs$,  $ad (Y) = Y \otimes _{\Rs} \Rs \cong Y \oplus \sigma^{\ast} Y$ as $\RsG$-module (\cite[Lemma 2.5.(i)]{Kn87}). 
The isomorphism gives the correspondence: 
$$
\Hom_{\RsG}(X, ad(Y))\cong \Hom_{\RsG}(X, Y \oplus \sigma^{\ast} Y) \qquad f \ \mapsto \ [x \mapsto (f_1(x), f_2(x))].
$$
Since $f(\sigma(x))= \sigma f(x)$, $(f_1(\sigma (x)), f_2(\sigma(x)))= \sigma (f_1(x), f_2(x)) = (f_2(x), f_1(x))$ (see Remark \ref{C2}). 
Thus $f_2 = f_1 \circ \sigma$. 
We determine the morphism $\Phi: \Hom_{\RsG} (X, Y \oplus \sigma ^{\ast} Y)\to \Hom_{\Rs}(\FF(X), Y)$ by $\phi (f) = f_1$. 
Conversely, by the observation above, we determine the morphism $\Psi: \Hom_{\Rs}(\FF(X), Y) \to \Hom_{\RsG} (X, Y \oplus \sigma ^{\ast} Y)$ by $\Psi (g) = (g, g \circ \sigma )$. 
By using the isomorphism $ad (Y)  \cong Y \oplus \sigma^{\ast} (Y)$, we see that $\Phi$ and $\Psi$ give natural isomorphisms between $\Hom_{\RsG} (X, ad(Y))$ and $\Hom_{\Rs}(\FF(X), Y)$. 
Compare with the proof of \cite[Theorem 3.2]{RR85}. 
\end{proof}

One can also show that $(\AA, \mathcal{B})$ and $(\mathcal{B}, \AA)$ are adjoint pairs between $\sCMR$ and $\sCMRs$ since the functors $\Omega$, $\FF$ and $ad$ preserve projectivity.   
Actually, $\Omega (R) = \syz _{\Rs} ^1 (R) \cong \Rs$ and $\Omega ^{-1} (\RsG) = (\RsG )^G / z (\RsG )^a \cong R$ where $(\RsG )^G$ (resp. $(\RsG )^a$) is the set of $\sigma$-invariant (resp. $\sigma$-antiinvariant) elements of $\RsG$(\cite[Proposition 2.1]{Kn87}). 
The indecomposable projective $\RsG$-modules are $\Rs$ and $\sigma ^{\ast} \Rs$ (\cite[Section 2]{Kn87}). 
This implies that $\FF$ sends projective $\RsG$-modules to projective (free) $\Rs$-modules since $\sigma ^{\ast} \Rs \cong \Rs$ as an $\Rs$-module. 
It also implies that the functor $ad$ sends projective (free) $\Rs$-modules to projective $\RsG$-modules.  
The fact induce that $(\AA, \mathcal{B})$ and $(\mathcal{B}, \AA)$ are also adjoint pairs between $\sCMR$ and $\sCMRs$. 
Hence we get the following consequence.

\begin{corollary}\label{C12}
Let $\AA$ and $\BB$ be as in Lemma \ref{C4}. 
Then $(\AA, \mathcal{B})$ and $(\mathcal{B}, \AA)$ are adjoint pairs between $\sCMR$ and $\sCMRs$. 
Thus the induced functors $\AA : \smodC \to \smodCs$ and $\BB : \smodCs \to \smodC$ are exact functors. 
\end{corollary}

\begin{remark}\label{C13}
According to Theorem \ref{C3}, for each $F \in \smodC$ with $\sHomR (-, Y) \xrightarrow{\sHomR (-, f)} \sHomR (-, X) \to F \to 0$, one has $\BB \circ \AA (F) = F \oplus  F[-1]$ defined by $\sHomR (-, Y) \oplus \sHomR (-, Y[-1]) \xrightarrow{\left(\begin{smallmatrix}\sHomR (-, f) &0\\0&\sHomR (-, f[-1])\end{smallmatrix}\right)} \sHomR (-, X)\oplus \sHomR (-, X[-1]) \to F \oplus F[-1] \to 0$. 
Similarly one also has $\AA \circ \BB (G) = G \oplus \sigma^{\ast} G$ for each $G \in \smodCs$ (see also Remark \ref{C2}).  
\end{remark}

\begin{proposition}\label{C8}
Let $R= S/(f )$ be a hypersurface and $\AA$, $\BB$ as in Lemma \ref{C4}. 
Suppose that, for each $F \in \smodC _{n}$, $\AA (F) \in \smodCs_{n}$. 
Then, for a simple object $S \in \smodC /\smodC _{n}$, $\AA (S)$ has finite length in $\smodCs /\smodCs _{n}$. 
\end{proposition}

\begin{proof}
Let $S$ be a simple object in $\smodC /\smodC_{n}$. 
If $\AA (S)$ is simple, we have nothing to prove. 
Thus we may assume that $\AA (S)$ is not simple. 
Then we have an exact sequence of functors $0 \to V \to \AA(S) \to S' \to 0$ such that $S'$ is simple  in $\smodCs/\smodCs_{n}$. 
Apply $\BB$ to the sequence, one has 
$$
0 \to \BB (V) \to \BB \circ \AA(S) \to \BB (S') \to 0. 
$$
Since $\BB \circ \AA(S) \cong S \oplus S[-1]$, $\ell (\BB (V)) + \ell (\BB (S')) = 2$ in $\smodC /\smodC_{n}$. 
Notice that the object $S[-1]$ is also a simple object (Lemma \ref{A7}).  
Suppose that $\ell (\BB (V)) = 0$, namely $\BB (V)$ belongs to $\smodC_{n}$. 
By the assumption $\AA \circ \BB (V) \cong V \oplus \sigma^{\ast} V$ is in $\smodCs_{n}$ and so is $V$. 
Hence $\AA (S) \cong S'$ in $ \smodC /\smodC _{n}$. 
This is a contradiction since $\AA (S)$ is not simple. 
Suppose that $\ell (\BB (S')) = 0$. 
Similarly one shows that $S'$ is in $\smodCs_{n}$, which is a contradiction. 
Consequently $\ell (\BB (V)) = \ell (\BB (S')) = 1$. 
Namely $\BB (V)$ and $\BB (S')$ are simple in $\smodC /(\smodC )_{n}$.

Now we shall show $V$ is also simple in $\smodCs /\smodCs_{n}$. 
Let $0 \to V' \to V \to C \to 0$ be an admissible sequence in $\smodCs$. 
Then we obtain the sequence $0 \to \BB (V') \to \BB (V) \to \BB (C) \to 0$ in $\smodC$.  
Since $\BB (V)$ is simple, $\BB (V' )$ or $\BB (C)$ is in $\smodC _{n}$. 
Assume that $\BB (V' )$ is in $\smodC _{n}$. 
By the assumption, $\AA \circ \BB (V')$ is in $\smodCs _{n}$, and so is $V'$. 
This implies that $V$ is simple in $\smodCs/\smodCs_{n}$. 
By the same arguments, the case that $\BB (C )$ is in $\smodC_{n}$ also implies that $V$ is simple.

Consequently $\AA (S)$ is of length 2 in $\smodCs/\smodCs_{n}$, so that the assertion holds. 
\end{proof}

\begin{proposition}\label{C9}
Suppose that $\AA (S)$ belongs to $\smodCs _{n}$ for each simple object $S$ in $\smodC /(\smodC)_{n-1}$. 
Then $\AA (F)$ is in $\smodCs_{n}$ for each $F$ in $\smodC_{n}$. 
\end{proposition}

\begin{proof}
We have the filtration of $F$ in $\smodC$
$$
0=F_0 \subset F_1 \subset F_2 \subset \cdots \subset F_n =F
$$
such that $F_i/F_{i-1}$ are simple in $\smodC /\smodC_{n-1}$. 
Apply $\AA$ to the filtration, we obtain the filtration
$$
0=\AA ( F_0 ) \subset \AA ( F_1 )\subset \AA ( F_2 ) \subset \cdots \subset \AA ( F_n) =\AA ( F)
$$
in $\smodCs$. 
By the assumption, $\AA (F_i )/ \AA (F_{i-1})$ is in $\smodCs _{n}$. 
Hence $\AA (F)$ belongs to $\smodCs_{n}$. 
\end{proof}

\begin{theorem}\label{C10}
Let $R= S/(f )$ and $\Rs = S[\![z]\!] /(f + z^2)$. 
Then $\KGdim \smodC =\KGdim \smodCs$. 
\end{theorem}

\begin{proof}
First we notice that, as mentioned in \cite[Corollary 2.10]{Kn87}, a simple object $S$ in $\smodC$ goes to a length-finite object in $\smodCs$, that is $\AA (S) \in \smodCs  _0$. 
Conversely a simple object $S'$ in $\smodCs$ also goes to an object in $\smodC _0$, that is $\BB (S') \in \smodC _0$.  
Summing up Proposition \ref{C8} and \ref{C9}, one can see that $\AA$ gives a functor from $\smodC _n$ to $\smodCs _n$. 
Suppose that $\KGdim \smodC =n$. 
For each object $F \in \smodCs$, $\BB (F)$ belongs to $\smodC _n$. 
Since $\AA \circ \BB (F) = F \oplus \sigma^{\ast}F$ belongs to $\smodCs _n$, $F$ is contained in $\smodCs _n$, so that $\KGdim F \leq n$. 
Therefore $\KGdim \smodCs \leq \KGdim \smodC$. 
Let $R^{\sharp \sharp} = S[\![z_1, z_2]\!] /(f + z_{1} ^2 + z_{2} ^2)$. 
Then we have an equivalence of triangulated categories $\sCM (R) \cong \sCM (R^{\sharp \sharp})$ which is known as Kn\"orrer's periodicity (cf. \cite[(12.10)]{Y}). 
Hence we also have the equivalence $\smodC \cong \smod (\CM(R^{\sharp \sharp}) )$. 
Apply the above arguments to $\smodCs$ and $\smodCss$, one has $\KGdim \smodC =\KGdim \smodCss \leq \KGdim \smodCs$. 
Consequently $\KGdim \smodC =\KGdim \smodCs$. 

Suppose that $\KGdim \smodC =\infty$. 
The inequality $\KGdim \smodC \leq \KGdim \smodCs$ holds if $\KGdim \smodCs$ is finite. 
Thus $\KGdim \smodCs$ must be infinite. 
\end{proof}

Finally we reach the main theorem of the paper.

\begin{corollary}\label{C11}
Let $k$ be an algebraically closed uncountable field of characteristic not two.
Let $R$ be a hypersurface of countable but not finite CM representation type. 
Then $\KGdim \smodC =2$. 
\end{corollary}

\begin{proof}
We may assume that $R$ is a hypersurface of type $(A_\infty)$ or $(D_\infty)$. 
Thanks to Theorem \ref{C10}, we have known that the Krull--Gabriel dimension is stable under Kn\"orrer's periodicity. 
Therefore the assertion holds from Theorem \ref{B1} and \ref{D3}. 
\end{proof}

\section*{Acknowledgments}
The author express his deepest gratitude to Ryo Takahashi and Yuji Yoshino for valuable discussions and helpful comments. 

%%%%%%%%%%%%%%%%%%%%%%%%%%%%% 
\ifx\undefined\bysame 
\newcommand{\bysame}{\leavevmode\hbox to3em{\hrulefill}\,} 
\fi

\end{document}